\documentclass[reqno,a4paper, 12pt]{amsart}

\usepackage[pdfpagelabels]{hyperref}
\usepackage{graphicx}

\usepackage[ansinew]{inputenc}
\usepackage{amsfonts,epsfig}
\usepackage{latexsym}
\usepackage{amsmath}
\usepackage{amssymb}

\newtheorem{theorem}{Theorem}
\newtheorem{lemma}[theorem]{Lemma}

\newtheorem{proposition}[theorem]{Proposition}

\theoremstyle{definition}

\theoremstyle{remark}

\numberwithin{equation}{section}

\newcommand{\set}[1]{\left\{#1\right\}}
\newcommand{\abs}[1]{\lvert#1\rvert}
\newcommand{\nm}[1]{\lVert#1\rVert}
\newcommand{\bnm}[1]{\big\Vert#1\big\Vert}

\newcommand{\D}{\mathbb{D}}

\newcommand{\N}{\mathbb{N}}

\newcommand{\C}{\mathbb{C}}

\renewcommand{\phi}{\varphi}

\DeclareMathOperator{\Area}{Area}

\renewcommand{\H}{\mathcal{H}}
\newcommand{\B}{\mathcal{B}}
\renewcommand{\tilde}{\widetilde}

\begin{document}

\title{Inner functions in weak Besov spaces}

\author{Janne Gr\"ohn}
\address{Department of Physics and Mathematics, University of Eastern Finland, P.O. Box 111, FI-80101 Joensuu, Finland}
\email[Corresponding author]{janne.grohn@uef.fi}

\author{Artur Nicolau}
\address{Departament de Matem\`atiques, Universitat Aut\`onoma de Barcelona, 08193, Bellaterra, Barcelona, Spain}
\email{artur@mat.uab.cat}

\thanks{The first author was supported by the Academy of Finland \#258125, and the second
author was supported in part by the projects MTM2011-24606 and 2009SGR420.}



\date{\today}


\begin{abstract}
It is shown that inner functions
in weak Besov spaces are precisely the exponential Blaschke products.
\end{abstract}

\maketitle



\section{Introduction}

Let $\H(\D)$ denote the algebra of holomorphic functions in the unit disc $\D$ of the complex plane $\C$.
Recall that $I\in\H(\D)$ is called inner provided that $I$ is  bounded in $\D$, and it satisfies
$\abs{I(e^{i\theta})} = 1$ for almost every (a.e.) $e^{i\theta}\in\partial\D$, where $I(e^{i\theta})=\lim_{r\to 1^-} I(re^{i\theta})$.
Every inner function $I$ can
be represented as a product $I = \xi B S$, where $\abs{\xi}=1$, $B$ is a Blaschke product,
and $S$ is a singular inner function; see \cite[p.~75]{G1981}.
It is a classical problem to classify those inner functions whose derivative belongs to a pre-given function space.

There is an extensive literature on inner functions whose derivative belongs to certain Hardy or Bergman space.
In the seventies Ahern, Clark and other authors proved many seminal results \cite{A1979,AC1974,AC1976,CS1969,C1971,P1973},
which have been extended and complemented later in various directions, see  for instance 
\cite{A1983,AJ1984,AV2010,D1994,D1998,D2006,FM2008,GGJ2011,GPV2007,J2009,K1984,M2013,PGR2009,P2008,P2004,P2011,V1982}.
The following presentation discusses some of those results. 
The first one is \cite[Theorem~1.1]{K1984}, see also \cite[Theorem~3.1]{DGV2002},
stating that the only inner functions in any Besov space
\begin{equation*}
  \B^p = \set{ f\in\H(\D) : \nm{f}_{B^p} = \left( \int_\D \abs{f'(z)}^p \, d\mu_p(z) \right)^{1/p} < \infty},
  \quad 1<p<\infty,
\end{equation*} 
are the finite Blaschke products. Here $d\mu_p(z) = (1-\abs{z})^{p-2} \, dA(z)$, and 
$dA(z)$ is the element of Lebesgue area measure in $\D$.

A Blaschke product $B$
is said to be exponential, if there exists a constant $M=M(B)$ with $0\leq M<\infty$ such that each annulus 
\begin{equation} \label{eq:annuli}
  \mathcal{A}_j = \set{z\in\D : 2^{-j} < 1 - \abs{z} \leq 2^{-j+1}}, \quad j\in\N,
\end{equation}
contains at most $M$ zeros of $B$. Exponential Blaschke products can be described
in terms of the growth of the integral means of their fractional derivatives \cite[Theorem~3.1]{AJ1984}, 
see also \cite{J2009, V1982}. Moreover, by \cite[Theorem~3]{AC1974} and \cite[Theorem~1]{CNpreprint}
exponential Blaschke products are the only inner functions whose derivative belongs to the weak 
Hardy space $H^1_w$, which consists of those functions $f\in\H(\D)$ for which
there exists a constant $C=C(f)$ with $0<C<\infty$ such that
\begin{equation*} 
  \big|\big\{ e^{i\theta} \in\partial\D : \abs{f(re^{i\theta})} > \lambda \big\}\big|  
  \leq C/\lambda, \quad 0<\lambda<\infty, \quad 0\leq r < 1.
\end{equation*}
Here $\abs{E}$ denotes the Euclidean length of the one-dimensional set $E$.
The present paper extends \cite[Theorem~1]{CNpreprint} by showing
that the inner functions in weak Besov spaces are precisely the exponential Blaschke products.


\subsection{Notation}\label{sec:notation}

We briefly recall the familiar concepts of $L^p$ spaces of arbitrary measure spaces, and then proceed to 
consider certain specific spaces of analytic functions in $\D$. Let $X$ be a measure space, and let $\mu$ 
be a positive measure on $X$. For $0<p<\infty$, $L^p(X,\mu)$ denotes the space of all
complex-valued $\mu$-measurable functions on $X$ whose modulus to the $p^\text{th}$ power is integrable.
The shorter notation $L^p(X)$ is reserved for the space $L^p(X,\mu)$ provided that $\mu$ is the Lebesgue measure
restricted to $X$. Correspondingly, $L^p_w(X,\mu)$ for $0<p<\infty$ is the weak $L^p(X,\mu)$ space, which contains 
those complex-valued $\mu$-measurable functions $f$ for which the quasi-norm
\begin{equation*}
  \nm{f}_{L^p_w(X,\mu)} 
  = \sup_{0<\lambda<\infty} \, \lambda  \, \mu\big( \big\{x\in X : \abs{f(x)}>\lambda\big\} \big)^{1/p}\\
\end{equation*}
is finite. Again, we write $L^p_w(X)$ instead of $L^p_w(X,\mu)$ in the case that $\mu$ is the Lebesgue measure
restricted to $X$. For $0<p<\infty$, we have $L^p(X,\mu)\subsetneq L^p_w(X,\mu)$.
If $0<r<p<\infty$, then the Kolmogorov condition \cite[p.~485, Lemma~2.8]{GGF1985} states that
$f\in L^p_w(X,\mu)$ if and only if there exists a constant $C=C(p,r,f)$ with $0<C<\infty$ such that
\begin{equation*}
\left( \int_E \abs{f}^r \, d\mu \right)^{\frac{1}{r}} \leq C \, \mu(E)^{\frac{1}{r}-\frac{1}{p}}
\end{equation*}
for all $\mu$-measurable subsets $E$ of $X$.
For more information on $L^p$ spaces of arbitrary measure spaces, we refer to \cite{GGF1985,G2008}.

Instead of working with arbitrary measure spaces, we concentrate on holomorphic functions in $\D$. 
As usual, the Hardy space $H^p$ for $0<p<\infty$ is defined as
\begin{equation*}
  H^p = \set{ f \in\H(\D) : \nm{f}_{H^p} 
     = \sup_{0\leq r < 1} \left( \frac{1}{2\pi} \int_0^{2\pi} \abs{f(re^{i\theta})}^p \, d\theta \right)^{1/p}< \infty}.
\end{equation*}
The non-tangential maximal function of $f$ is given by
\begin{equation*} 
  M_\sphericalangle f(e^{i\theta})
  = \sup_{z\in \Gamma_\alpha(e^{i\theta})} \abs{f(z)}, \quad e^{i\theta}\in\partial\D,
\end{equation*}
for fixed $1<\alpha<\infty$, where the involved non-tangential region (Stolz angle)
\begin{equation} \label{eq:stolz}
  \Gamma_\alpha(e^{i\theta}) = \set{ z\in\D :  \abs{z-e^{i\theta}} \leq \alpha(1-\abs{z})}, \quad e^{i\theta}\in\partial\D,
\end{equation}
has a vertex at $e^{i\theta}\in\partial\D$ with aperture equal to $2 \arctan\sqrt{\alpha^2 -1}$. 
A fundamental result by Hardy and Littlewood (for $1< p < \infty$) and by Burkholder, 
Gundy and Silverstein (for $0 < p \leq 1$)  states that Hardy spaces can be characterized by means 
of the non-tangential maximal function. That is, if $f\in\H(\D)$ and $0<p<\infty$, then
$f\in H^p$ if and only if $M_\sphericalangle f\in L^p(\partial\D)$. See \cite[p.~57, Theorem~3.1]{G1981}.

For $0<p<\infty$, the weak Hardy space $H^p_w$ is defined as
\begin{equation*}
  H^p_w = \set{f\in\H(\D) : 
    \nm{f}_{H^p_w} = \sup_{0\leq r < 1} \, \nm{f_r}_{L^p_w(\partial\D)} < \infty}, \quad f_r(z) = f(rz).
\end{equation*}
The weak Hardy spaces can be also characterized in terms of the non-tangential maximal function; 
if $f\in\H(\D)$ and $0<p<\infty$, then $f\in H^p_w$ if and only if $M_\sphericalangle f\in L^p_w(\partial\D)$.
We refer to \cite[p.~36]{CMR2006} for further discussion.


\subsection{Weak Bergman spaces}

Define
\begin{equation*}
  \mathcal{L}^p_w = \H(\D) \cap  L^p_w(\D,\mu_p), \quad 1<p<\infty.
\end{equation*}
Since $\mu_1$ is not finite in $\D$, we employ a different definition to correspond the case $p=1$. 
We write $f\in\tilde{\mathcal{L}}^1_w$ provided that $f\in\H(\D)$ and $(1-\abs{z})^{-1} \abs{f(z)} \in L^1_w(\D,A)$.

The nesting property of the spaces $\mathcal{L}^p_w$, which is given by Proposition~\ref{prop:inclusion}(ii), 
relies on a certain point-wise growth estimate. If $0\leq p<\infty$, then we define the growth spaces 
$A^{-p}$, and $A^{-p}_0$, to consist of functions $f\in\H(\D)$ satisfying
\begin{equation*}
  \nm{f}_{A^{-p}} = \sup_{z\in\D} \, (1-\abs{z})^p \abs{f(z)} < \infty, \quad \text{and} \quad \lim_{\abs{z}\to 1^-} \, (1-\abs{z})^p \abs{f(z)} =0,
\end{equation*}
respectively; see \cite{K1975}.


\begin{proposition} \label{prop:inclusion}
We have
\begin{enumerate}
\item[(i)] 
  $H^1 \subset \tilde{\mathcal{L}}^1_w \subset H^1_w \subset \mathcal{L}^p_w \subset A^{-1}$ for all $1<p<\infty$;

\item[(ii)] 
  $\mathcal{L}^p_w \subset \mathcal{L}^q_w$ for all $1<p<q<\infty$.
\end{enumerate}
\end{proposition}

It is worth mentioning that $\mathcal{L}^p_w$ is strictly larger than $H^1_w$ for any $1<p<\infty$. For instance,
since $L^p(\D,\mu_p)$ contains the Bloch space, $\mathcal{L}^p_w$ contains functions in $\H(\D)$ having
finite radial limit at no point of the unit circle for any $1<p<\infty$. In regard to Proposition~\ref{prop:inclusion}(i), we 
point out that $H^1\subset A^{-1}_0$ \cite[Theorem~5.9]{D1970}.


\subsection{Derivatives of inner functions}

The weak Besov spaces $\B^p_w$ are defined~as
\begin{equation*}
  \B^p_w = \big\{ f\in\H(\D) : f' \in L^p_w(\D,\mu_p) \big\}, \quad 1<p<\infty.
\end{equation*}
Our main result characterizes inner functions in weak Besov spaces.
This should be compared with the result mentioned in the introduction, 
according to which the only inner functions in the corresponding Besov spaces are the finite Blaschke products.


\begin{theorem} \label{thm:weakBergman}
Let $B$ be an inner function, and let $1<p<\infty$. Then,
$B\in\B^p_w$ if and only if $B$ is an exponential Blaschke product.
\end{theorem}

In view of \cite[Theorem~1]{CNpreprint} and Proposition~\ref{prop:inclusion}(i),
our contribution in Theorem~\ref{thm:weakBergman} is to prove that,
if $B$ is inner and $B'\in\mathcal{L}^p_w$ for some $1<p<\infty$, then $B$ is an exponential Blaschke product.
The following result characterizes inner functions whose derivative belongs to~$\tilde{\mathcal{L}}^1_w$,
which is a space intermediate to $H^1$ and $H^1_w$ by Proposition~\ref{prop:inclusion}(i).


\begin{theorem} \label{thm:p1}
Let $B$ be an inner function. Then, $B'\in\tilde{\mathcal{L}}^1_w$ if and only if $B$ is a~finite Blaschke product.
\end{theorem}

The subsequent sections are devoted for the proofs of new results.


\section{Proof of Proposition~\ref{prop:inclusion}}


\subsection{Proof of Proposition~\ref{prop:inclusion}(i)}

We proceed to prove each inclusion separately.


\subsubsection*{Proof of $H^1 \subset \tilde{\mathcal{L}}^1_w$}

Let $f\in H^1$. This implies that $M_\sphericalangle f\in L^1(\partial\D)$,
and hence by Chebyshev's inequality
\begin{equation} \label{eq:che}
  \lambda \, \big| \big\{ e^{i\theta} \in\partial\D \,: \, M_\sphericalangle f(e^{i\theta}) > \lambda \big\} \big|
  \leq \bnm{M_\sphericalangle f}_{L^1(\partial\D)} < \infty,
  \quad 0 < \lambda < \infty.
\end{equation}

We proceed to show that $f\in\tilde{\mathcal{L}}^1_w$. Suppose that $2<\lambda<\infty$ is given. Then,
\begin{align}
  & \Area \big\{z \in\D : \abs{f(z)} > \lambda (1-\abs{z}) \big\}\notag\\
  & \qquad = \int_0^{1-\lambda^{-1}} \big| \big\{ e^{i\theta} \in\partial \D  : \abs{f(re^{i\theta})}> \lambda ( 1 - r) \big\} \big| \,  r\, dr \label{eq:e1}\\
  & \qquad \quad +  \int_{1-\lambda^{-1}}^1 \big| \big\{ e^{i\theta} \in\partial \D  :  \abs{f(re^{i\theta})}> \lambda ( 1 - r) \big\} \big| 
  \,  r\, dr.\label{eq:e2}
\end{align}
Let $I_1$ and $I_2$ be the integrals in  \eqref{eq:e1} and \eqref{eq:e2}, respectively.
To compute $I_1$, let $K=K(\lambda)$ be the largest natural number such that $K\leq \log_2 \lambda$. This implies that
$1-2^K\lambda^{-1} < 1/2$, and hence by \eqref{eq:che} we deduce
\begin{align}
  I_1 
  & \leq \int_0^{1/2} \left| \left\{ e^{i\theta} \in\partial \D  :   M_\sphericalangle f(e^{i\theta}) > \lambda/2 \right\} \right| \, dr\notag\\
  & \quad + \sum_{k=1}^K \, \int_{1-2^{k} \lambda^{-1}}^{1 - 2^{k-1} \lambda^{-1}} 
  \left| \left\{ e^{i\theta} \in\partial \D  :   M_\sphericalangle f(e^{i\theta}) >2^{k-1} \right\} \right| \, dr\notag\\
  & \leq \frac{1}{\lambda} \, \bnm{M_\sphericalangle f}_{L^1(\partial\D)}
  + \frac{1}{\lambda} \, \sum_{k=1}^K 
  2^{k-1} \left| \left\{ e^{i\theta} \in\partial \D  :   M_\sphericalangle f(e^{i\theta}) >2^{k-1} \right\} \right|. \label{eq:defS}
\end{align}
Let $S$ be the sum in \eqref{eq:defS}. By the summation by parts, and \eqref{eq:che},  we obtain
\begin{align*}
  S & = \sum_{k=0}^{K-1} 
  \left( 2^{k+1} - 2^k \right) \left| \left\{ e^{i\theta} \in\partial \D  :   M_\sphericalangle f(e^{i\theta}) >2^k \right\} \right|\\
  & =  \left| \left\{ e^{i\theta} \in\partial \D  :   M_\sphericalangle f(e^{i\theta}) >2^K \right\} \right| 2^K 
  -  \left| \left\{ e^{i\theta} \in\partial \D  :   M_\sphericalangle f(e^{i\theta}) > 1 \right\} \right| \\ 
  &  \quad - \sum_{k=0}^{K-1} 2^{k+1} \Big(  \left| \left\{ e^{i\theta} \in\partial \D  :   M_\sphericalangle f(e^{i\theta}) >2^{k+1} \right\} \right|
  - \left| \left\{ e^{i\theta} \in\partial \D  :   M_\sphericalangle f(e^{i\theta}) >2^k \right\} \right| \Big)\\
  &  \leq \bnm{M_\sphericalangle f}_{L^1(\partial\D)}
  + \sum_{k=0}^{K-1} 2^{k+1}  \left| \left\{ e^{i\theta} \in\partial \D  :   2^k < M_\sphericalangle f(e^{i\theta}) \leq 2^{k+1} \right\} \right|.
\end{align*}
It follows that
\begin{equation*}
  S  \leq \bnm{M_\sphericalangle f}_{L^1(\partial\D)}
  + 2 \, \sum_{k=0}^{K-1}
  \int_{\left\{ e^{i\theta} \in\partial \D  \, : \,   2^k < M_\sphericalangle f(e^{i\theta}) \leq 2^{k+1} \right\}} M_\sphericalangle f(e^{i\theta})\, d\theta
   \leq  3 \, \bnm{M_\sphericalangle f}_{L^1(\partial\D)}.
\end{equation*}

Since $I_2 \leq 2\pi/\lambda$, we conclude that $f\in\tilde{\mathcal{L}}^1_w$.


\subsubsection*{Proof of $\tilde{\mathcal{L}}^1_w \subset H^1_w$}

If $f\in\tilde{\mathcal{L}}^1_w$, then
there exists a constant $C=C(f)$ with $0<C<\infty$ such that
\begin{equation*}
  \Area \big\{ z\in\D: \abs{f(z)} > \lambda (1-\abs{z}) \big\} \leq C/\lambda, \quad 0<\lambda<\infty.
\end{equation*}
Write $r_k=1-2^{-k}$ for $k\in\N$, for short. Now
\begin{align*}
  \frac{C}{\lambda 2^k} 
  & \geq \Area\set{z\in\D: \abs{f(z)} > \lambda 2^k (1-\abs{z})}\\
  & = \int_0^1 \big|\big\{e^{i\theta}\in\partial\D  :  \abs{f(re^{i\theta})} > \lambda 2^k (1-r) \big\}\big|\,  r \, dr\\
  & \geq \frac{1}{2} \int_{r_k}^{r_{k+1}} \big|\big\{e^{i\theta}\in\partial\D  : \abs{f(re^{i\theta})} > \lambda \big\}\big|\, dr,
  \quad 0<\lambda<\infty, \quad k\in\N.
\end{align*}
Since $r_{k+1}-r_k= 2^{-k-1}$ for all $k\in\N$, we conclude that
\begin{equation} \label{eq:target}
  \frac{1}{r_{k+1}-r_k} \, \int_{r_k}^{r_{k+1}} \big|\big\{e^{i\theta}\in\partial\D  : \abs{f(re^{i\theta})} > \lambda \big\}\big|\, dr
  \leq \frac{4C}{\lambda}, \quad 0<\lambda<\infty,
\end{equation}
for all $k\in\N$.

Fix $0<p<1$. By using the distribution function, Fubini's theorem implies that
\begin{align*}
  &  \frac{1}{r_{k+1}-r_k} \, \int_{r_k}^{r_{k+1}}  \left( \int_0^{2\pi} \abs{f(re^{i\theta})}^p \, d\theta \right) dr\\
  & \qquad \leq \frac{1}{r_{k+1}-r_k} \, \int_{r_k}^{r_{k+1}}  
  \left( 2\pi + \int_1^\infty 
    \big| \big\{e^{i\theta} \in\partial\D :  \abs{f(re^{i\theta})}>\lambda\big\} \big|
    \, p \lambda^{p-1} \, d\lambda \right)  dr\\
  & \qquad = 2\pi + \int_1^\infty  \left( \frac{1}{r_{k+1}-r_k} \, \int_{r_k}^{r_{k+1}}  
    \big| \big\{e^{i\theta} \in\partial\D  :  \abs{f(re^{i\theta})}>\lambda\big\} \big|
    \, dr \right) \, p \lambda^{p-1} \, d\lambda
\end{align*}
for all $k\in\N$. By \eqref{eq:target} there exists a constant $K=K(p,f)$ with $0<K<\infty$ such that
\begin{equation} \label{eq:hp}
  \frac{1}{r_{k+1}-r_k} \, \int_{r_k}^{r_{k+1}}  \left( \int_0^{2\pi} \abs{f(re^{i\theta})}^p \, d\theta \right)  dr
   \leq K <\infty, \quad k\in\N.
\end{equation}
Since the mapping $r \mapsto \nm{f_r}_{L^p(\partial\D)}^p$ is non-decreasing \cite[p.~9]{D1970}, 
\eqref{eq:hp} yields
\begin{equation*}
  \int_0^{2\pi} \abs{f(r_k e^{i\theta})}^p \, d\theta \leq K < \infty, \quad k\in\N,
\end{equation*}
from which we deduce that $f\in H^p$.

Suppose that $0<\lambda<\infty$ is given.  Since $f$ belongs to $H^p$ for $0<p<1$, 
the radial limit $f(e^{i\theta})= \lim_{r\to 1^-} f(re^{i\theta})$ exists for a.e. $e^{i\theta}\in\partial\D$.
By applying Egorov's theorem \cite[p.~73]{R1987}, we conclude
that this convergence is uniform outside a set of arbitrarily small Lebesgue measure.
In particular, there exist a constant $r^\star=r^\star(\lambda,f)$ with $0<r^\star<1$, 
and a set $E=E(\lambda,f)\subset \partial\D$ 
of length $\abs{E}\leq 1/\lambda$, such that
\begin{equation*}
  \big| f(r e^{i\theta})-f(e^{i\theta}) \big| \leq \lambda/2, \quad r^\star < r< 1, \quad e^{i\theta}\in \partial\D\setminus{E}.
\end{equation*} 
Observe that  $\set{e^{i\theta} \in\partial\D: \abs{f(e^{i\theta})} > \lambda} \subset A \cup B$, 
where $A=A(\lambda,f)$ and $B=B(\lambda,f)$ are subsets of $\partial\D$ given by
\begin{align*}
  A & = \big\{e^{i\theta} \in\partial\D : 
    \text{$\big|f(re^{i\theta}) - f(e^{i\theta})\big| > \lambda/2$ for some $r$ with $r^\star <r<1$}\big\},\\
  B & = \big\{e^{i\theta} \in\partial\D : \text{$\abs{f(e^{i\theta})} > \lambda$, $\big|f(re^{i\theta}) - f(e^{i\theta}) \big|
    \leq \lambda/2$ for all $r$ with $r^\star <r<1$}\big\}.
\end{align*}
Since $A\subset E$, we conclude that $\abs{A}\leq 1/\lambda$.
We proceed to consider the size of~$B$. Note that
$B \subset \set{e^{i\theta}\in\partial\D : \abs{f(re^{i\theta})} > \lambda/2}$ for any $r$ with $r^\star<r<1$. Consequently,
if $k\in\N$ is sufficiently large such that $r^\star \leq r_k<1$, then \eqref{eq:target} implies that
\begin{equation*}
  \abs{B} 
  \leq \frac{1}{r_{k+1}-r_k} \int_{r_k}^{r_{k+1}} \big| \big\{e^{i\theta}\in\partial\D : \abs{f(re^{i\theta})} > \lambda/2 \big\} \big| \, dr
   \leq \frac{8C}{\lambda}.
\end{equation*}
In conclusion,
  $\big| \big\{ e^{i\theta}\in\partial\D : \abs{f(e^{i\theta})} > \lambda \big\} \big| 
  \leq \abs{A} + \abs{B}
  \leq (1+8C)/\lambda$,
where $C$ is independent of $\lambda$. This proves that $f(e^{i\theta}) \in L^1_w(\partial\D)$. Since $f\in H^p$
for $0<p<1$, we have $f\in H^1_w$ by \cite[Theorem~1.10.4]{CMR2006}.


\subsubsection*{Proof of $H^1_w \subset \mathcal{L}^p_w$ for  $1<p<\infty$}

Let $f\in H^1_w$, and hence $M_\sphericalangle f\in L^1_w(\partial\D)$, 
and let $1<p<\infty$ be fixed. Suppose that $\lambda_0<\lambda<\infty$ is given, where 
\begin{equation*}
  \lambda_0 
  = \lambda_0(\alpha,f)
  = \max\left\{ \sup\set{\abs{f(z)} : \abs{z}\leq \frac{\alpha-1}{\alpha+1}}, 
    \frac{\bnm{M_\sphericalangle f}_{L^1_w(\partial\D)}}{2\sqrt{\alpha^2-1}} \right\}.
\end{equation*}
Here $\alpha>1$ is the (fixed) constant in \eqref{eq:stolz}, which determines the aperture
of the non-tangential region.
Since $\set{ e^{i\theta}\in\partial\D : M_\sphericalangle f(e^{i\theta}) > \lambda}$ is an open subset of $\partial\D$, 
there exists a family of pairwise disjoint open arcs $I_k = I_k(\lambda,f) \subset\partial\D$ for $k\in\N$,
such that $\bigcup_{k\in\N} I_k = \set{ e^{i\theta} \in\partial\D: M_\sphericalangle f(e^{i\theta}) > \lambda}$, and
\begin{equation} \label{eq:sumlambda}
  \sum_{k\in\N} \, \abs{I_k} 
  \leq \frac{\bnm{M_\sphericalangle f}_{L^1_w(\partial\D)}}{\lambda}.
\end{equation}

For each arc $I\subset\partial\D$ we define a corresponding (sectorial) domain $T(I)$ by
\begin{equation*}
  T(I) = \set{z\in\D : \text{$z/\abs{z} \in I$ and } 1 - \abs{z} \leq \frac{\abs{I}}{2\sqrt{\alpha^2-1}}}.
\end{equation*}
Since $\lambda>\lambda_0\geq \sup\big\{ \abs{f(z)} : \abs{z}\leq (\alpha-1)/(\alpha+1)\big\}$, we obtain
\begin{equation*} 
  \big\{z\in\D : \abs{f(z)} > \lambda\big\} \subset \bigcup_{k\in\N} \, T(I_k),
\end{equation*}
which yields
$\mu_p\big( \big\{z\in\D : \abs{f(z)} > \lambda\big\} \big) \leq \sum_{k\in\N} \mu_p \big( T(I_k) \big)$.
Since there exists a constant $C=C(\alpha,p)$ with $0<C<\infty$, such that
\begin{equation*} 
  \mu_p\big( T(I_k) \big)  = \int_{T(I_k)} (1-\abs{z})^{p-2} \, dA(z)
  \leq C \, \abs{I_k}^p, \quad k\in\N,
\end{equation*}
we deduce that
\begin{equation*}
  \mu_p\big( \big\{z\in\D : \abs{f(z)} > \lambda\big\} \big)
  \leq C \, \sum_{k\in\N} \, \abs{I_k}^p
  \leq C \left( \, \sum_{k\in\N} \, \abs{I_k} \right)^p
  \leq C \, \frac{\bnm{M_\sphericalangle f}_{L^1_w(\partial\D)}^p}{\lambda^p}
\end{equation*}
by \eqref{eq:sumlambda}. Hence $f\in\mathcal{L}^p_w$.


\subsubsection*{Proof of $\mathcal{L}^p_w\subset A^{-1}$ for $1<p<\infty$}

Suppose that $f\in \mathcal{L}^p_w$ for some $1<p<\infty$. 
Let $a\in\D$, and suppose that $E_a=D(a,(1-\abs{a})/2)$ is a Euclidean disc of radius $(1-\abs{a})/2$.
The Kolmogorov condition mentioned in Section~\ref{sec:notation} (with $r=1$) shows that
there exists a constant $C_1=C_1(p,f)$ with $0<C_1<\infty$ such that
\begin{equation*}
  \int_{E_a} \abs{f(z)} (1-\abs{z})^{p-2} \, dA(z) \leq C_1 \left( \, \int_{E_a} (1-\abs{z})^{p-2} \, dA(z) \right)^{1-1/p},
  \quad a\in\D.
\end{equation*}
By the subharmonicity of $\abs{f}$, and the Kolmogorov condition above,
there exists a~constant $C_2=C_2(p,f)$ with $0<C_2<\infty$, such that
\begin{align*}
  \abs{f(a)} 
  & \leq \frac{4}{\pi (1-\abs{a})^2} \int_{E_a} \abs{f(z)} \, dA(z)\\
  & \leq \frac{2^p}{\pi (1-\abs{a})^p} \, \int_{E_a} \abs{f(z)} (1-\abs{z})^{p-2} \, dA(z)
  \leq \frac{C_2}{(1-\abs{a})}, \quad a\in\D,
\end{align*}
and hence $f\in A^{-1}$.


\subsection{Proof of Proposition~\ref{prop:inclusion}(ii)}

Suppose that $f\in \mathcal{L}^p_w$ for some $1<p<\infty$, and let $p<q<\infty$. We proceed to prove that $f\in\mathcal{L}^q_w$.
Let $2 < \lambda < \infty$ be given, and take $K=K(\lambda)\in\N$ such that $2^K< \lambda \leq 2^{K+1}$. Now
\begin{equation*}
  \mu_q \big( \big\{ z\in\D : \abs{f(z)} > \lambda \big\} \big)
  \leq \sum_{k=K}^\infty \, \int_{\{z\in\D \, :\,  2^{k+1} \geq \abs{f(z)} > 2^{k}\}} (1-\abs{z})^{q-2}\, dA(z).
\end{equation*}

According to Proposition~\ref{prop:inclusion}(i) we may suppose that $f\in A^{-1}$. Consequently,
if $k\in\N$ and $2^k < \abs{f(z)}$, then $1-\abs{z} < 2^{-k} \, \nm{f}_{A^{-1}} $. We conclude that
\begin{align*}
  \mu_q \big( \big\{ z\in\D : \abs{f(z)} > \lambda \big\} \big)
  & \leq \nm{f}_{A^{-1}}^{q-p} \sum_{k=K}^\infty 2^{kp-kq} \, \int_{\{z\in\D \, :\,  \abs{f(z)} > 2^{k}\}} (1-\abs{z})^{p-2}\, dA(z)\\
  & \leq \frac{\nm{f}_{A^{-1}}^{q-p} \, \nm{f}_{L^p_w(\D,\mu_p)}^p}{1-2^{-q}}\, \frac{2^q}{\lambda^q}.
\end{align*}
This means that $f\in \mathcal{L}^q_w$, and we are done.


\section{Proof of Theorem~\ref{thm:weakBergman}}

The point of departure is a discussion of certain auxiliary results, which show
that inner functions whose derivative belongs to $\mathcal{L}^p_w$ for some $1<p<\infty$
reduce to finite products of interpolating Blaschke products.


\begin{proposition} \label{prop:B_Blaschke}
If $B$ is inner, and if $B'\in \mathcal{L}^p_w$ for some $1<p<\infty$, then $B$ is a~Blaschke product.
\end{proposition}


\begin{proof}
By the assumption $B'\in L^p_w(\D,\mu_p)$, and consequently $B' \in L^q(\D,\mu_p)$ for any $0<q<p$.
That is,
\begin{equation*}
  \int_{\D} \abs{B'(z)}^q (1-\abs{z})^{p-2} \, dA(z)<\infty, \quad 0<q<p.
\end{equation*}
If we choose $q$ to satisfy $p-1/2\leq q<p$, then we conclude that $B$ is a Blaschke product by \cite[p.~736]{A1983}.  
\end{proof}

A Blaschke product $B$ is called interpolating provided that its zeros $\{z_k\}_{k\in\N}$ form
a uniformly separated sequence in $\D$. That is to say that there exists a constant $\delta=\delta(B)$ with
$0<\delta<1$ such that
\begin{equation} \label{eq:unisep}
  \inf_{k\in\N} \, \prod_{n\in\N\setminus\set{k}} \left| \frac{z_k-z_n}{1-\overline{z}_k z_n} \right|
  = \inf_{k\in\N} \, (1-\abs{z_k}^2) \abs{B'(z_k)} \geq \delta.
\end{equation}


\begin{proposition} \label{prop:B_finprod}
If $B$ is a Blaschke product, and if $B'\in \mathcal{L}^p_w$ for some $1<p<\infty$, then
$B$ is a finite product of interpolating Blaschke products.
\end{proposition}

Before the proof of Proposition~\ref{prop:B_finprod} we consider a standard lemma,
whose proof bears similarity to that of \cite[Lemma~1]{GN1996}. 
The proof of Lemma~\ref{lemma:needed} is given for the convenience of the reader.
For $\zeta\in\D$, let
\begin{equation*}
  \Delta(\zeta,m)= \big\{ z\in\D : \abs{z-\zeta} < m \, \abs{1-\overline{z}\zeta} \big\}
\end{equation*}
be an open pseudo-hyperbolic disc of the pseudo-hyperbolic radius $0<m<1$.


\begin{lemma} \label{lemma:needed}
Suppose that $B$ is a Blaschke product, which is not a finite product of interpolating Blaschke products.
Then there exist sequences $\set{\zeta_k}_{k\in\N}\subset \D$, and $\set{m_k}_{k\in\N}\subset (0,1)$, such that 
\begin{equation} \label{eq:suplimit}
  \lim_{k\to\infty} m_k = 1, \quad \text{and} \quad \lim_{k\to\infty} \left( \sup_{z\in \Delta(\zeta_k,m_k)} \, \abs{B(z)} \right) = 0.
\end{equation}
\end{lemma}


\begin{proof}
Since $B$ is not a finite product of interpolating Blaschke products, we know that
$\mu = \sum_{k\in\N} (1-\abs{z_k})\, \delta_{z_k}$
is not a Carleson measure \cite[Lemma~21]{MS1979}. Here $\delta_{z_k}$ is the Dirac measure with unit point mass
at $z_k\in\D$, and $\set{z_k}_{k\in\N}$ is the zero-sequence of $B$. Consequently, there exists a
sequence $\set{Q_j}_{j\in\N}$ of Carleson boxes $Q_j=Q_j(B)$ of the form
\begin{equation*}
  Q_j  = \big\{z\in\D : \text{$z/\abs{z}\in I_{Q_j}$ and $1-\abs{z} \leq \abs{I_{Q_j}}$}\big\}, \quad j\in\N,
\end{equation*}
where $I_{Q_j}$ is an arc on $\partial\D$ of length $\abs{I_{Q_j}} = l(Q_j)$, 
such that the corresponding sequence $\set{S_j}_{j\in\N}$, where $S_j=S_j(B)\in (0,\infty)$ is defined by
\begin{equation*}
  S_j = \frac{1}{l(Q_j)} \, \sum_{z_k\in Q_j} (1-\abs{z_k}) = \frac{\mu(Q_j)}{l(Q_j)},
\end{equation*}
satisfies $S_j\to\infty$, as $j\to\infty$.
Without loss of generality, we may suppose that $Q_j \cap \set{z_k}_{k\in\N} \neq \emptyset$,
and $l(Q_j)\leq 1/2$, for every $j\in\N$.

Let $z(Q_j)\in\D$ be the point such that $\abs{z(Q_j)} = 1-l(Q_j)$, and whose argument is the center of $I_{Q_j}$.
Define $m_j$ by $1-m_j = S_j^{-1/2}$ for $j\in\N$. Since $S_j\to\infty$, we conclude that
$m_j\to 1^-$, as $j\to\infty$. If $z\in\Delta\big(z(Q_j),m_j\big)$, then
$1-\abs{z} > (1-m_j)\, l(Q_j)/8$ by \cite[p.~42]{DS2004}, which implies that
\begin{align*}
  \log\frac{1}{\abs{B(z)}} 
  & \geq \frac{1}{2} \, \sum_{z_k \in Q_j} \frac{(1-\abs{z}^2)(1-\abs{z_k}^2)}{\abs{1-\overline{z}_k z}^2}
  \geq \frac{1-\abs{z}^2}{18\cdot l(Q_j)^2} \, \sum_{z_k \in Q_j} (1-\abs{z_k}^2)\\
  & \geq \frac{1-\abs{z}^2}{18\cdot l(Q_j)} \cdot \frac{\mu(Q_j)}{l(Q_j)}
  \geq \frac{1-m_j}{144} \cdot S_j = \frac{S_j^{1/2}}{144} \to \infty, \quad j\to\infty,
\end{align*}
for all $z\in\Delta\big(z(Q_j),m_j\big)$.
Here we have applied the inequality $\log 1/x \geq (1-x^2)/2$ for $0<x<1$.
The assertion follows by choosing $\set{\zeta_k}_{k\in\N} = \set{z(Q_j)}_{j\in\N}$,
and $\set{m_k}_{k\in\N} = \set{m_j}_{j\in\N}$.
\end{proof}


\begin{proof}[Proof of Proposition~\ref{prop:B_finprod}]
Suppose on the contrary to the claim that $B$ is a Blaschke product, which
is not a finite product of interpolation Blaschke products.
By Lemma~\ref{lemma:needed} there exist $\set{\zeta_k}_{k\in\N}\subset \D$, 
and $\set{m_k}_{k\in\N}\subset (0,1)$, such that  \eqref{eq:suplimit} holds.

Let $1<M<\infty$ be a large fixed number, whose exact value is to be determined later.
Let $I_k\subset\partial\D$ for $k\in\N$ be the arc, 
which is centered at $\zeta_k/\abs{\zeta_k}$, and whose
length is $\abs{I_k} = 1-\abs{\zeta_k}$. After removing finitely many points from $\set{\zeta_k}_{k\in\N}$
we may assume that $\abs{I_k}\leq 1/2$ for all $k\in\N$.
Correspondingly, $M I_k$ denotes the concentric arc of length $M \abs{I_k} = M(1-\abs{\zeta_k})$. Define
\begin{equation*}
  E_k = \big\{ re^{i\theta} \in\D : \text{$e^{i\theta}\in M I_k$, and $1-r\leq \abs{I_k}$} \big\}, \quad k\in\N.
\end{equation*}

On one hand, \eqref{eq:suplimit} shows that $\abs{B(z)}\leq 1/2$ holds for all $z\in\D$ such
that $\abs{z}=1-\abs{I_k}$, and $z/\abs{z} \in M I_k$, provided that $k$ is sufficiently large.
Thus
\begin{equation*}
\frac{1}{2} \leq \big| B(e^{i\theta}) - B\big( (1-\abs{I_k}) e^{i\theta} \big) \big|
      \leq \int_{1-\abs{I_k}}^1 \abs{B'(re^{i\theta})} \, dr, \quad \text{a.e. $e^{i\theta} \in M I_k$},
\end{equation*}
and consequently $\int_{E_k} \abs{B'(z)} \, dA(z) \geq M \abs{I_k}/4$ for all $k\in\N$
large enough.

On the other hand, if $\max\set{1,p-1}<r<p$, then the Kolmogorov condition
stated in Section~\ref{sec:notation}
shows that there exists a constant $C_1=C_1(p,r,B)$ with $0<C_1<\infty$ such that
\begin{equation*}
  \left( \int_{E_k} \abs{B'(z)}^r (1-\abs{z})^{p-2} \, dA(z) \right)^{\frac{1}{r}} \leq C_1 \left( \int_{E_k} (1-\abs{z})^{p-2} \, dA(z) \right)^{\frac{1}{r}-\frac{1}{p}},
  \quad k\in\N.
\end{equation*}
Now H\"older's inequality with indices $r$ and $r/(r-1)$ implies that there exists a~constant
$C_2=C_2(p,r)$ with $0<C_2<\infty$ such that
\begin{align*}
  \int_{E_k} \abs{B'(z)} \, dA(z)
  & \leq \left( \int_{E_k} \abs{B'(z)}^r (1-\abs{z})^{p-2} \, dA(z) \right)^{\frac{1}{r}}
  \left( \int_{E_k} (1-\abs{z})^\frac{2-p}{r-1} \, dA(z) \right)^{\frac{r-1}{r}}\\
  & \leq C_1 \left( \int_{E_k} (1-\abs{z})^{p-2} \, dA(z) \right)^{\frac{1}{r}-\frac{1}{p}}
  \left( \int_{E_k} (1-\abs{z})^\frac{2-p}{r-1} \, dA(z) \right)^{\frac{r-1}{r}}\\
  & \leq C_1 \,  C_2\, M^{1-\frac{1}{p}} \, \abs{I_k}, \quad k\in\N.
\end{align*}

We conclude that, if $1<M<\infty$ is sufficiently large, then the obtained upper and lower bounds for
$\int_{E_k} \abs{B'(z)} \, dA(z)$ lead to a contradiction. The assertion follows.
\end{proof}

Finally we are in a position to prove our main result.


\begin{proof}[Proof of Theorem~\ref{thm:weakBergman}]
Since derivatives of exponential Blaschke products are in $H^1_w$ by \cite[Theorem~1]{CNpreprint},
Proposition~\ref{prop:inclusion}(i) implies that it suffices to show that,
if $B$ is inner and $B'\in \mathcal{L}^p_w$ for some $1<p<\infty$, 
then $B$ is an exponential Blaschke product.

By Propositions~\ref{prop:B_Blaschke} and~\ref{prop:B_finprod}, we may assume that 
$B$ is a finite product of interpolating Blaschke products.  Observe that, if $\set{z_k}_{k\in\N}$ is the zero-sequence of $B$,
then it is possible that \eqref{eq:unisep} fails to be true for any $0<\delta<1$ due to zeros with multiplicities. 
However, according to \cite[Lemma~1]{GN1996} there exists a sequence $\set{\zeta_k}_{k\in\N}$, 
where $\zeta_k=\zeta_k(B)\in\D$ for $k\in\N$, as well as constants $\rho=\rho(B)$ and $\delta=\delta(B)$ satisfying
$0<\rho<1$ and $0<\delta\leq 1$ such that
\begin{equation} \label{eq:point}
  \inf_{k\in\N} \, (1-\abs{\zeta_k}^2) \abs{B'(\zeta_k)} \geq \delta,
\end{equation}
while $\zeta_k \in \Delta(z_k,\rho)$ for all $k\in\N$. By the Schwarz lemma, inequality \eqref{eq:point} can be
extended to sufficiently small pseudo-hyperbolic neighborhoods of the points $\zeta_k$ for $k\in\N$. 
In particular, there exist constants $\eta=\eta(B)$ and $\delta^\star=\delta^\star(B)$ satisfying $0<\eta<1$ and $0<\delta^\star \leq 1$,
such that
\begin{equation} \label{eq:newe}
  (1-\abs{z}^2) \abs{B'(z)} \geq \delta^\star, \quad z\in \Delta(\zeta_k,\eta), \quad k\in\N.
\end{equation}

Since $B$ is a finite product of interpolating Blaschke products, we can write
$B=B_1 B_2 \dotsb B_N$, where the zero-sequence $\set{z_{n,k}}_{k\in\N}$ of each sub-product $B_n$ for $n=1,\dotsc,N$
satisfies \eqref{eq:unisep} for some strictly positive constant $\delta=\delta_n$. By taking smaller $\eta$ if necessary,
we may assume that $\eta$ in \eqref{eq:newe} is sufficiently small to satisfy
$\eta < \min\set{\delta_1,\delta_2,\dotsc,\delta_N}/2$ and $\eta<1/3$. Consequently, the pseudo-hyperbolic discs
$\set{\Delta(z_{n,k},\eta)}_{k\in\N}$ are pairwise disjoint for any $n=1,\dotsc,N$. Although
the discs $\set{\Delta(z_k,\eta)}_{k\in\N}$ are not necessarily pairwise disjoint, they are quasi-disjoint
in the sense that the characteristic functions $\chi_{\Delta(z_k,\eta)}$ of the discs $\Delta(z_k,\eta)$ for $k\in\N$ satisfy
\begin{equation*}
  \sum_{k\in\N} \chi_{\Delta(z_k,\eta)}(z) \leq N, \quad z\in\D.
\end{equation*}
This implies that the sequence $\set{\zeta_k}_{k\in\N}$ in \eqref{eq:point} satisfies
\begin{equation} \label{eq:quasidis}
  \sum_{k\in\N} \chi_{\Delta(\zeta_k,\eta)}(z) \leq M, \quad z\in\D,
\end{equation}
for some constant $M=M(B)$ with $0<M<\infty$, and hence the discs $\set{\Delta(\zeta_k,\eta)}_{k\in\N}$ are quasi-disjoint.

Let $\mathcal{A}_j$ for $j\in\N$ be the annuli in \eqref{eq:annuli}, and observe 
that their pseudo-hyperbolic widths are always greater than $1/3$.
Take $\lambda$ such that $\delta^\star\leq \lambda<\infty$,
and denote $J(\lambda) = \log_2 \lambda/\delta^\star +3$.
If $k\in\N$ such that $\zeta_k\in\mathcal{A}_j$ for some $j> J(\lambda)$, then
$\Delta(\zeta_k,\eta) \subset \big( \mathcal{A}_{j-1} \cup \mathcal{A}_j \cup \mathcal{A}_{j+1} \big)$,
and by \eqref{eq:newe}
\begin{equation*}
  \abs{B'(z)} > \frac{\delta^\star/2}{1-\abs{z}} \geq \delta^\star 2^{j-3} > \lambda, \quad z\in\Delta(\zeta_k,\eta).
\end{equation*}

Define $\mathcal{N}_j = \# \big( \mathcal{A}_j \cap \set{\zeta_k}_{k\in\N} \big)$ 
for $j\in\N$. Since $B'\in \mathcal{L}^p_w$, and the discs $\set{\Delta(\zeta_k,\eta)}_{k\in\N}$ 
satisfy \eqref{eq:quasidis}, we obtain
\begin{align*}
  \frac{\bnm{B'}^p_{L^p_w(\D,\mu_p)}}{\lambda^p} 
  & \geq \mu_p\big( \big\{ z\in\D : \abs{B'(z)} > \lambda\big\} \big)\\
  & \geq  \frac{1}{3} \, \sum_{j>J(\lambda)} \,  \mu_p\big( \big\{ z\in \big( \mathcal{A}_{j-1} \cup \mathcal{A}_j \cup \mathcal{A}_{j+1} \big) 
      :  \abs{B'(z)} > \lambda \big\} \big) \\
  & \geq  \frac{1}{3} \, \sum_{j>J(\lambda)} \,  \mu_p \Bigg( \bigcup_{\zeta_k\in\mathcal{A}_j} \Delta(\zeta_k,\eta) \Bigg)
  \geq \frac{1}{3M} \, \sum_{j>J(\lambda)}  \sum_{\zeta_k\in\mathcal{A}_j} \,\mu_p \big( \Delta(\zeta_k,\eta) \big)\\
  & =  \frac{1}{3M} \, \sum_{j>J(\lambda)}  \sum_{\zeta_k\in\mathcal{A}_j} \,\int_{\Delta(\zeta_k,\eta)} (1-\abs{z})^{p-2}\, dA(z).
\end{align*}
We deduce
\begin{equation*}
  \frac{\bnm{B'}^p_{L^p_w(\D,\mu_p)}}{\lambda^p} \geq  \frac{C}{M} \, \sum_{j>J(\lambda)}  \mathcal{N}_j \, 2^{-jp},
\end{equation*}
where $C=C(\eta,p)$ is a constant such that $0<C<\infty$.
Consequently, each annulus $\mathcal{A}_j$ for $j\in\N$ contains at most a fixed number of points 
from the sequence $\set{\zeta_k}_{k\in\N}$.
This means that $B$ is an exponential Blaschke product, since 
$\zeta_k\in\Delta(z_k,\rho)$ for all $k\in\N$.
This closes the proof of Theorem~\ref{thm:weakBergman}.
\end{proof}


\section{Proof of Theorem~\ref{thm:p1}}

It is sufficient to prove that, if $B$ is inner and $B'\in\tilde{\mathcal{L}}^1_w$,
then $B$ is a finite Blaschke product. To this end,
suppose that $B$ is an inner function such that $B'\in \tilde{\mathcal{L}}^1_w$, and assume on the
contrary to the claim that $B$ has infinitely many zeros. 
According to Proposition~\ref{prop:inclusion}(i)  we know that $B'\in H^1_w$, and
hence $B$ is an exponential Blaschke product \cite[Theorem~1]{CNpreprint}. Consequently, $B'$ has radial
limits almost everywhere. If $\set{z_k}_{k\in\N}$ is the zero-sequence of $B$, then
\begin{equation*}
  \abs{B'(\xi)} = \sum_{k\in\N} \frac{1-\abs{z_k}^2}{\abs{\xi-z_k}^2},
  \quad \text{a.e. $\xi\in\partial\D$},
\end{equation*}
by \cite[Corollary~3]{AC1974}. Let $I_k$ for $k\in\N$ be 
the arc on $\partial\D$, which is centered at $z_k/\abs{z_k}$, and
whose length is $2\pi(1-\abs{z_k})$. If $z_k\in\D$ is a zero of $B$, then
\begin{equation*}
  \abs{B'(\xi)} \geq \frac{1-\abs{z_k}^2}{\abs{\xi-z_k}^2} > \frac{1}{(1+\pi)^2} \, \frac{1}{1-\abs{z_k}},
\quad \text{a.e. $\xi\in I_k$}.
\end{equation*}
By Egorov's theorem \cite[p.~73]{R1987}, there exists a sequence $\set{E_k}_{k\in\N}$, where $E_k$ are subsets of
$I_k$ satisfying $\abs{E_k} > \abs{I_k}/2$ for all $k\in\N$, 
and a sequence $\set{r_k}_{k\in\N}\subset (0,1)$, such that
\begin{equation*}
  \abs{B'(r \xi)} > \frac{1}{(1+\pi)^2} \, \frac{1}{1-\abs{z_k}},
  \quad r_k<r<1, \quad \xi\in E_k.
\end{equation*}

By extracting a subsequence of $\set{z_k}_{k\in\N}$,  and by choosing the sets $E_k$ accordingly, 
we may assume that the sets $H_k=\big\{r\xi\in\D : r_k < r <1, \, \, \xi\in E_k \big\}$ for $k\in\N$ 
are pairwise disjoint, while $\abs{E_k}>\abs{I_k}/4$. Fix $0<\lambda<\infty$, and define
\begin{equation*}
  K(\lambda) = \left\{ k\in\N : \lambda > \frac{1}{(1+\pi)^2(1-r_k)(1-\abs{z_k})} \right\}.
\end{equation*}
Observe that $K(\lambda)$ is a bounded subset of $\N$, while
$\#K(\lambda)\to\infty$, as $\lambda\to\infty$. Now
\begin{equation*}
  \left\{ z\in\D : \frac{\abs{B'(z)}}{1-\abs{z}}> \lambda \right\} \supset
  \bigcup_{k\in K(\lambda)} \left\{ z\in H_k : \frac{1}{(1+\pi)^2 (1-\abs{z_k})(1-\abs{z})} > \lambda \right\},
\end{equation*}
which implies that
\begin{align*}
  & \Area \left\{ z\in\D : \frac{\abs{B'(z)}}{1-\abs{z}}> \lambda \right\}\\
  & \qquad \geq \sum_{k\in K(\lambda)} \Area \left\{ z\in H_k : \frac{1}{(1+\pi)^2 (1-\abs{z_k})(1-\abs{z})} > \lambda \right\}
  \geq \frac{\#K(\lambda) \cdot \pi}{2(1+\pi)^2\lambda}.
\end{align*}
Consequently $B' \notin \tilde{\mathcal{L}}^1_w$, which closes the proof.


\section*{Acknowledgements}

The authors thank K. Dyakonov for helpful conversations.
The first author is also grateful to the Universitat Aut\`onoma de Barcelona 
for hospitality during the academic year 2012-2013.


\end{document}